\documentclass[12pt]{article}
\usepackage[utf8]{inputenc}
\usepackage[T1]{fontenc}
\usepackage{lmodern}
\usepackage{amsthm}
\usepackage{amsfonts}
\usepackage{amsmath}
\usepackage{amssymb}
\usepackage{colonequals}
\usepackage{epsfig}
\usepackage{url}
\usepackage{tikz}
\usetikzlibrary{calc}
\usepackage{todonotes}

\usepackage{mathtools,microtype}

\newcommand{\GG}{{\cal G}}

\newtheorem{theorem}{Theorem}
\newtheorem{corollary}[theorem]{Corollary}
\newtheorem{lemma}[theorem]{Lemma}

\newtheorem{question}{Question}

\DeclareMathOperator{\Pee}{\mathbf{Pr}}
\DeclareMathOperator{\col}{col}
\DeclareMathOperator{\fan}{fan}

\newenvironment{tikzgraph}
  {\begin{tikzpicture}
      [vertex/.style={circle, draw=black, fill=black, inner sep=0.5pt, minimum
        size=6pt},
       edge/.style={thick},%
      ]\begin{scope}}
  {\end{scope}\end{tikzpicture}}

\title{On generalized choice and coloring numbers\thanks{This work falls within the scope of L.I.A. STRUCO.}}
\author{Zden\v{e}k Dvo\v{r}\'ak\thanks{Charles University, Prague, Czech Republic.
E-mail: {\tt rakdver@iuuk.mff.cuni.cz}.  Supported by project 17-04611S (Ramsey-like aspects of graph
coloring) of Czech Science Foundation.}
\and Jakub Pek\'arek\thanks{Charles University, Prague, Czech Republic.
E-mail: {\tt pekarej@iuuk.mff.cuni.cz}.}
\and Jean-S\'ebastien Sereni\thanks{Centre National de la Recherche Scientifique, CSTB (ICube), Strasbourg, France.
E-mail: \texttt{sereni@kam.mff.cuni.cz}. This work was partially supported by A.N.R. Project STINT and P.H.C. Barrande~40625WH.}}
\date{}

\begin{document}
\maketitle

\begin{abstract}
A well-known result of Alon shows that the coloring number of a graph is bounded by a function
of its choosability.  We explore this relationship in a more general setting with relaxed assumptions
on color classes, encoded by a graph parameter.
\end{abstract}

There exist countless variations of proper graph colorings, where the constraints on the structure
of the color classes are either relaxed, stronger or simply different. In other words, instead of
requiring color classes to be independent sets, one can require them to have maximum degree,
or tree-width, or component sizes bounded from above by a fixed parameter. This article contributes
to an effort toward unifying our understanding of such variants of graph coloring.

A coloring of a graph~$G$ is \emph{proper} if adjacent vertices receive distinct colors, and the \emph{chromatic number}~$\chi(G)$ of~$G$
is the least integer~$s$ for which~$G$ admits a proper coloring using~$s$ different colors.
A \emph{list assignment} for~$G$ is a function~$L$ that to each vertex assigns a set of colors.
It is an \emph{$s$-list assignment} if $|L(v)|\ge s$ for each vertex~$v\in V(G)$.
An \emph{$L$-coloring} is a coloring~$\varphi$ of~$G$ such that~$\varphi(v)\in L(v)$ for all~$v\in V(G)$.
An $L$-coloring is \emph{proper} if no two adjacent vertices have the same color.
The \emph{choosability}~$\chi^\ell(G)$ of~$G$ is the least integer~$s$ such that~$G$ has a proper $L$-coloring for every
$s$-list assignment~$L$. The \emph{coloring number}~$\col(G)$ of~$G$ is the least integer~$s$
such that every subgraph of~$G$ contains a vertex of degree less than~$s$.  A straightforward greedy argument shows
that $\chi^\ell(G)\le \col(G)$. While the gap between the chromatic number and the coloring number can be arbitrarily large
--- there are $\Delta$-regular bipartite graphs for every integer~$\Delta$ ---
Alon~\cite{Alo00} proved that the same is not true regarding choosability:
the coloring number of a graph is bounded by an exponential function of its choosability, which can be seen
as a weak converse of the previous upper bound.

One can equivalently define a proper coloring~$\varphi$ as one in which, for every color~$c$, its color class $\varphi^{-1}(c)$
induces an independent set in~$G$.  A number of relaxations of this concept have been studied, requiring instead that the color classes induce
subgraphs with bounded maximum degree~\cite{CCW86,CGJ97b,CuKi10,FrHe94,HaSe06,KaMc10,Skr99b,Woo04}, bounded maximum component size~\cite{ADOV03,EsOc16,HST03} or bounded tree-width~\cite{BDN18,DDO+04}, for instance.
This suggests the following generalization, proposed by Dvo\v r\'ak and Norin~\cite{DvNo17}.  Let~$f$ be a graph parameter, assigning to every
graph an element of~$\mathbb{N}\cup\{\infty\}$,
such that isomorphic graphs are assigned the same value.  For an integer~$p$, a coloring of
a graph~$G$ is \emph{$(f,p)$-proper} if $f(G[\varphi^{-1}(c)])\le p$ for each
color~$c$.  We can now naturally define~$\chi_{f,p}(G)$ as the least number~$s$ of
colors in an $(f,p)$-proper coloring of~$G$ and $\chi^{\ell}_{f,p}(G)$ as the
least integer~$s$ such that~$G$ has an $(f,p)$-proper $L$-coloring for every
$s$-list assignment~$L$ of~$G$, or~$\infty$ if no such integer~$s$ exists.
For example, if $f(G)=\Delta(G)$ then $\chi_{f,p}$ is the defective
chromatic number with defect~$p$ as introduced by Cowen, Cowen and Woodall~\cite{CCW86}.
Regarding defective colorings (sometimes called improper colorings) as well as clustered
colorings (which both fall in the scope of our work), the reader is referred to the recent
survey of Wood~\cite{Woo17}.

An analogue to the coloring number relevant in this context was introduced by Esperet and Ochem~\cite{EsOc16}.
Given a positive integer~$s$, an \emph{$s$-island} in a graph~$G$
is a non-empty subset~$I$ of vertices of~$G$ such that each vertex in~$I$ has less than~$s$ neighbors (in~$G$) not belonging to~$I$.
We define~$\col_{f,p}(G)$ as the least integer~$s$ such that for every induced subgraph~$H$ of~$G$, the graph~$H$ contains an $s$-island~$I$ satisfying 
$f(H[I])\le p$.  In particular, for every graph~$G$ we have $\col_{f,p}(G)=1$ if and only if~$f(H)\le p$ for every connected induced subgraph~$H$ of~$G$.
Under reasonable assumptions on~$f$, the invariant~$\col_{f,p}(G)$ is an upper bound on~$\chi^{\ell}_{f,p}(G)$.

A parameter~$f$ is \emph{hereditary} if $f(H)\le f(G)$ whenever~$H$ is an induced subgraph of~$G$,
and~$f$ is \emph{connected} if $f(G)=\max(f(G_1),f(G_2))$ whenever~$G$ is the disjoint union of the graphs~$G_1$ and~$G_2$.
\begin{lemma}\label{lemma:colboundsch}
If~$f$ is a connected and hereditary parameter, then $\chi^{\ell}_{f,p}(G)\le \col_{f,p}(G)$ for every graph~$G$ and every integer~$p$.
\end{lemma}
\begin{proof}
Let us set~$s \coloneqq \col_{f,p}(G)$ and fix an $s$-list assignment~$L$ for~$G$.
We proceed by induction on the number of vertices of~$G$.
If~$G$ is the null graph (with no vertices), then the statement is trivially true.
Otherwise let~$I$ be an $s$-island in~$G$ satisfying $f(G[I])\le p$ and let~$G' \coloneqq G-I$. It
follows from the induction hypothesis that $\chi^\ell_{f,p}(G')$ is at most~$\col_{f,p}(G')$,
which itself is at most~$s$ by the definition of $\col_{f,p}$. Hence, there exists an
$(f,p)$-proper $L$-coloring~$\psi$ of~$G'$. We extend~$\psi$ to~$G$ by
assigning to each vertex in~$I$ an arbitrary color from its list that is not
used on its neighbours that belong to~$V(G')$, each vertex in~$I$ having at least one
such color since $I$ is an $s$-island. We observe that this yields an $(f,p)$-proper $L$-coloring
of~$G$. Indeed, let~$H$ be a subgraph of~$G$ induced by the vertices colored
with an arbitrary color~$c$.  It follows from the way~$\psi$ was extended
that every connected component of~$H$ is either contained in~$I$ or disjoint
from~$I$.  Since~$f$ is connected, $f(H)=\max\{f(H[I]),f(H - I)\}$. The definition
of~$\psi$ ensures that $f(H - I) \leq p$ and since~$f$ is hereditary, $f(H[I])$ is at most~$f(G[I])$,
which itself is at most~$p$.  This concludes the proof.
\end{proof}

It is natural to ask whether a converse result, analogous to that of Alon for ordinary list coloring, holds
for this generalization.  More precisely, could it be the case that all parameters~$f$ (subject to some natural assumptions)
have the following property?
\begin{itemize}
\item[(CC)] For all integers~$p$ and~$s$, there exist two integers~$p'$ and~$s'$ such that
      $\col_{f,p'}(G)\le s'$ for every graph~$G$ satisfying $\chi^{\ell}_{f,p}(G)\le s$.
\end{itemize}
\begin{figure}
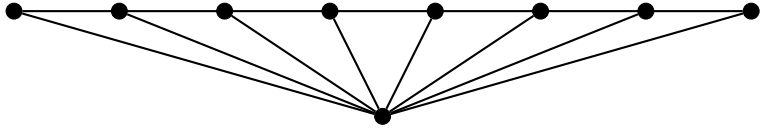
 
      \begin{center}
            \begin{tikzgraph}
\def\l{14mm}
\foreach \i in {0,...,7}
  \draw (\i*\l,\l) node[vertex](u\i){};

\draw (3.5*\l,0) node[vertex](v){};

\foreach \i in {0,...,7} 
  \draw[edge] (v)--(u\i);

\draw[edge] (u0)--(u7);
\end{tikzgraph}
      \end{center}
      \caption{A fan on nine vertices.}\label{fig-1}
\end{figure}
The answer turns out to be ``no'', even for an arguably reasonable graph parameter defined as follows.
A \emph{fan} is a graph consisting of a path and another vertex adjacent to all its vertices (see Figure~\ref{fig-1}).
For a graph~$G$, let~$\fan(G)$ be the maximum number of vertices of a fan appearing in~$G$ as a subgraph.
The parameter~$\fan$ is monotone (i.e., $\fan(H)\le\fan(G)$ for every subgraph~$H$ of~$G$) and connected.
Nevertheless, the following lemma shows that the parameter~$\fan$ does not satisfy~(CC).
The \emph{full join} of two graphs~$G$ and~$H$ is obtained from the disjoint union of~$G$ and~$H$ by adding an edge between every vertex of~$G$ and every vertex of~$H$.
\begin{lemma}\label{lemma:nofan}
      There exists a sequence~$(G_i)_{i\in\mathbb{N}}$ of graphs such that
$\chi^{\ell}_{\fan,2}(G_i)\le 2$ and $\col_{\fan,i}(G_i)\ge i+1$
      for every positive integer~$i$.
\end{lemma}
\begin{proof}
Let~$G_i$ be the full join of a path~$P_i$ and an independent set~$A_i$,
where~$P_i$ has~$i^2$ vertices and~$A_i$ has~$i$ vertices.  Given any $2$-list
assignment~$L$ for~$P_i$, it is possible to $L$-color~$P_i$ so that no edge is monochromatic.
This clearly prevents any monochromatic fan on more than two vertices
in~$G_i$, regardless of the coloring of~$A_i$.  Hence, $\chi^{\ell}_{\fan,2}(G_i)\le 2$.

It now suffices to show that every $i$-island of~$G_i$ contains a fan with more than~$i$ vertices.
Let~$I$ be a non-empty set of vertices of~$G_i$ such that $\fan(G_i[I]) \leq i$. We show that $I$ is not an $i$-island.
If~$I$ contains no vertex from~$A_i$, then
every vertex in~$I$ has at least~$i$ neighbours in~$V(G_i)\setminus I$
and therefore~$I$ is not an $i$-island. Now let~$v$ be a vertex in~$I \cap A_i$. Note that $I$ does not contain a set of~$i$ consecutive vertices from~$P_i$, as otherwise these vertices together with~$v$ would form a fan of~$G_i$, contradicting that
$\fan(G_i[I])\leq i$. Consequently, since $V(P_i)$ can be partitioned into $i$ vertex-disjoint sets of~$i$ consecutive vertices, we conclude that at least~$i$ vertices of~$P_i$ do not belong to~$I$.
Each of these~$i$ vertices is adjacent to~$v$, showing that $I$ is not an $i$-island. 
\end{proof}
While we were unable to fully describe the graph parameters satisfying~(CC), we provide two sufficient conditions
for a graph parameter to satisfy~(CC). These two conditions basically cover all parameters that have been studied in this setting.

One way to establish an upper bound on~$\col_{f,p}(G)$ is to show that the ordinary coloring number of~$G$ is bounded by some integer~$s$.
Then, each subgraph of~$G$ contains a vertex of degree less than~$s$ forming an $s$-island by itself, and thus if~$p$ is the value
of~$f$ on the single-vertex graph, it follows that $\col_{f,p}(G)\le s$.  This motivates the study of the following stronger property of a graph parameter~$f$.
\begin{itemize}
\item[(CD)] For all integers~$p$ and~$s$, there exists an integer~$s'$ such that $\col(G)\le s'$ for every graph~$G$
      satisfying $\chi^{\ell}_{f,p}(G)\le s$.
\end{itemize}
We characterize hereditary graph parameters having property~(CD) exactly. A graph parameter~$f$ \emph{bounds the average degree}
if there exists some function~$g\colon\mathbb{N}\to\mathbb{N}$ such that every graph~$G$ has average degree at most~$g(f(G))$.
\begin{theorem}\label{thm:cdchar}
A hereditary graph parameter satisfies~(CD) if and only if it bounds the average degree.
\end{theorem}
By Theorem~\ref{thm:cdchar}, all the parameters that were mentioned before (namely, maximum degree, component size and treewidth)
have property~(CD), and thus also property~(CC).

Another graph parameter that comes to mind is the chromatic number.  Although it does not bound average degree
and consequently does not have property (CD), it does have property (CC), for a fairly trivial reason which is explained
by the following lemma, in a slightly more general setting --- and requiring only that $\chi_{f,p}$ is bounded, not
$\chi^\ell_{f,p}$.
\begin{lemma}\label{lemma:addit}
Let~$f$ be a hereditary graph parameter.  Assume furthermore that there exists a function~$g\colon\mathbb{N}^2\to \mathbb{N}$
such that  $f(G)\le g(f(G[X]),f(G-X))$ for every graph~$G$ and every set~$X\subseteq V(G)$.
Then for all integers~$p$ and~$s$, there exists an integer~$p'$ such that $\col_{f,p'}(G)=1$ for every graph~$G$
      satisfying $\chi_{f,p}(G)\le s$.
\end{lemma}
\begin{proof}
Let~$\psi\colon V(G)\to\{1,\dotsc,s\}$ be an $(f,p)$-proper coloring of~$G$.
      For every~$i\in\{1,\dotsc,s\}$,
      let~$H_i$ be the subgraph of~$G$ induced by the vertices colored~$i$,
      and let~$G_i$ be the subgraph of~$G$ induced by the vertices in~$\psi^{-1}(1)\cup\dotsb\cup\psi^{-1}(i)$; thus,
      $V(G_i)=V(H_1)\cup\dotsb\cup V(H_i)$.
      Since~$\psi$ is $(f,p)$-proper, $f(H_i) \leq p$ for every~$i\in\{1,\dotsc,s\}$. Our assumption on~$f$ implies that
      $f(G_i) \leq g(f(H_i),f(G_{i-1}))$ whenever~$i\in\{2,\dotsc,s\}$. Moreover, $f(G_1) \leq p$ and~$G_s = G$.
      Therefore $f(G) \leq g(p,g(p,\dots ,g(p,p)\dots))$ where the depth of recursion is~$s-1$. This implies
      that $\col_{f,p'}(G)=1$ if~$p'$ is the right side of the previous inequality. 
\end{proof}
Note that $\chi(G)\le \chi(G[X])+\chi(G-X)$, and thus Lemma~\ref{lemma:addit} applies when~$f$ is the chromatic number.

The bound on the coloring number ($s'$) from the property~(CD) that we obtain in the proof of Theorem~\ref{thm:cdchar} of course depends
on both~$s$ and~$p$.  In the property~(CC), we have the additional freedom of choosing the parameter~$p'$, and one might hope that
by choosing~$p'$ large enough depending on~$p$ and~$s$, it could be possible to obtain a bound on the number of colors ($s'$)
that depends only on~$s$ (or possibly even have $s'=s$).  This is not the case, as we show in Section~\ref{sec:cluster}
in the important special case of clustered coloring.  Before doing that, let us prove Theorem~\ref{thm:cdchar}.

\section{Condition~(CD) and parameters bounding the average degree}\label{sec:cdchar}

As we mentioned before, Alon~\cite{Alo00} proved that the coloring number of a graph is bounded
by a function of its choosability.  Theorem~\ref{thm:cdchar} is a consequence of the following strengthening of this
result.  For a graph~$H$, let~$\mathrm{mad}(H)$ be the maximum of the average degrees of subgraphs of~$G$.
In Theorem~\ref{thm:minav}, we prove that for any~$k$, the coloring number of a graph~$G$ is bounded by a function of $\chi^\ell_{\mathrm{mad},k}(G)$;
that is, we show that for any~$k$ and~$s$, if~$G$ is a graph of large enough minimum degree (compared to~$k$ and~$s$),
then there exists an $s$-list-assignment~$L$ such that any $L$-coloring of~$G$ contains a
monochromatic subgraph of average degree greater than~$k$.
We note that a similar strenghtening of Alon's result was obtained by Kang~\cite[Theorem~6]{Kan13} in the context of
defective colorings---he proved that for any~$k$, the coloring number of a graph~$G$ is bounded by a function of $\chi^\ell_{\Delta,k}(G)$.
Since $\mathrm{mad}(G)\le \Delta(G)$, this result is implied by our Theorem~\ref{thm:minav}; however, let us note that
our argument can be obtained from Kang's by a minor modification and that both arguments are a suitable modification of Alon's.

We need the following corollary of Chernoff's inequality (see for instance the book
by Frieze and Karo\'nski~\cite[Theorem~21.6, p.~414]{FrKa16}).

\begin{lemma}\label{lemma:chernoff}
Let~$X_1, \dotsc, X_n$ be independent random variables taking values in~$[0,1]$.
      Set $X=\sum_{i=1}^n X_i$, and let~$\mu$ be the expected value of~$X$.  Then
      \[\Pee[X\le \mu/2]\le e^{-\mu/8}.\]
\end{lemma}

We now establish two lemmas (Lemmas~\ref{lemma-exab} and~\ref{lemma:probsmall}),
which place us in a position to obtain the statement (Theorem~\ref{thm:minav}) desired
to prove Theorem~\ref{thm:cdchar}.

\begin{lemma}\label{lemma-exab}
Let~$s$ and~$k$ be positive integers, and let~$G$ be a graph of minimum degree at least~$d$.
Set~$n\coloneqq|V(G)|$ and~$S\coloneqq\{1,\dotsc,s^2\}$.  If $d\ge 2^{14}k^2s^44^s$, then there exist disjoint sets~$A,B\subset V(G)$
and a function~$L_0$ that assigns to each vertex of~$B$ a subset of~$S$ of size~$s$, such that the following
holds:
\begin{itemize}
\item[(a)] $|A|\ge n/2$;
\item[(b)] $|B|\le 2n/\sqrt{d}$; and
\item[(c)] for every set~$T\subseteq S$ of size~$\lceil s^2/2\rceil$, each vertex in~$A$ has
at least~$ks^2$ neighbors~$b$ in~$B$ such that $L_0(b)\subset T$.
\end{itemize}
\end{lemma}

\begin{proof}
We consider a random subset~$B$ of~$V(G)$ obtained by placing each vertex~$v$ of~$V(G)$
in~$B$ with probability~$\frac{1}{\sqrt{d}}$, independently for each vertex.
Whenever a vertex~$v$ is placed in~$B$, we also choose at random a set~$L_0(v)$
uniformly among the subsets of~$S$ of size~$s$, independently for each such vertex.
By the linearity of expectation, the expected size of~$B$ is~$n/\sqrt{d}$, and by Markov's inequality,
the set~$B$ has size at most~$2n/\sqrt{d}$ with probability at least~$1/2$.

A vertex~$v\in V(G)$ is \emph{good} if firstly $v\notin B$ and secondly
for every subset~$T$ of~$S$ of size~$\lceil s^2/2\rceil$, the vertex~$v$
has at least~$ks^2$ neighbors~$b$ in~$B$ such that $L_0(b)\subset T$.
Note that for each vertex~$v$ and each fixed subset~$T\in\binom{S}{\lceil s^2/2\rceil}$,
the probability that a given neighbor of~$v$ belongs to~$B$ and has a list
contained in~$T$ is exactly
\[
p=\frac{1}{\sqrt{d}}\cdot\frac{\lceil s^2/2\rceil(\lceil s^2/2-1\rceil)\dotso(\lceil s^2/2\rceil -s+1)}{s^2(s^2-1)\dotso(s^2-s+1)}.
\]
It follows that the expected number of such neighbors of~$v$ is $p\cdot\deg_G(v)$.
Since all random choices are independent, Chernoff's bound (Lemma~\ref{lemma:chernoff}) ensures that the probability
that~$v$ has less than $p\cdot\deg_G(v)/2$ such neighbors is at most~$e^{-p\deg_G(v)/8}$.
Alon~\cite{Alo00} noted that
\begin{align}
\frac{\lceil s^2/2\rceil(\lceil s^2/2-1\rceil)\dotso(\lceil s^2/2\rceil -s+1)}{s^2(s^2-1)\dotso(s^2-s+1)}
      &\ge2^{-s}\prod_{i=0}^{s-1}\frac{s^2-2i}{s^2-i}\nonumber\\
      &=2^{-s}\prod_{i=0}^{s-1}\left(1-\frac{i}{s^2-i}\right)\nonumber\\
      &\ge2^{-s}\left(1-\frac{\sum_{i=0}^{s-1}i}{s^2-s}\right)\nonumber\\
      &=2^{-s-1}.\label{eq-estim}
\end{align}
This gives us $p\ge \frac{1}{\sqrt{d}}2^{-s-1}$, and thus
$p\deg(v)\ge pd\ge \sqrt{d}\cdot2^{-s-1}>2ks^2$. 
So we infer that
the probability that~$v$ has less than~$ks^2$ neighbors~$u$ in~$B$ with $L_0(u)\subset T$
is at most~$e^{-pd/8}\le e^{-\sqrt{d}\cdot2^{-s-4}}$. 
Consequently, for each vertex~$v\in V(G)$ the probability that~$v$ is
not good is at most
\[
      \frac{1}{\sqrt{d}}+\left(1-\frac{1}{\sqrt{d}}\right)\binom{s^2}{\lceil s^2/2\rceil}e^{-\sqrt{d}\cdot2^{-s-4}}.
\]
Since $\binom{s^2}{\lceil s^2/2\rceil}\le 2^{s^2}$, the probability
that an arbitrary vertex is not good is at most
\[
\frac{1}{\sqrt{d}}+2^{s^2}e^{-\sqrt{d}\cdot 2^{-s-4}}<\frac{1}{4}
\]
      by the assumptions on~$d$.
The expected number of vertices that are not good is thus less than~$\frac{n}{4}$, and
hence Markov's inequality yields that the probability that there are at least~$n/2$ non-good vertices
is less than~$1/2$.  Consequently, with probability greater than~$1/2$ there are at least~$n/2$
good vertices. It follows that there is a positive probability that simultaneously
$|B|\le 2n/\sqrt{d}$ and the number of good vertices is at least~$n/2$.

We fix a choice of~$B$ and of lists for the vertices in~$B$ such that $B$ has size at
most~$2n/\sqrt{d}$ and the set~$A$ of good vertices has size at least~$n/2$.  Hence, the
conditions (a) and (b) are satisfied, and (c) holds since $A$ consists of good vertices.
\end{proof}

Suppose that we are in the situation described in the statement of Lemma~\ref{lemma-exab}.
Let~$L_1$ be an assignment of lists to the vertices in~$A$.
Let~$\varphi$ be an $L_0$-coloring of~$G[B]$.  We define~$A_{\varphi,L_1}$
as the set of vertices~$v\in A$ such that for every color~$c\in L_1(v)$,
the vertex~$v$ has at least~$k$ neighbors~$u$ in~$B$ with~$\varphi(u)=c$.

\begin{lemma}\label{lemma:probsmall}
Let~$s$ and~$k$ be positive integers, and let~$G$ be a graph of minimum degree at least~$d$, where $d\ge 2^{14}k^2s^44^s$.
Set~$n\coloneqq|V(G)|$ and~$S\coloneqq\{1,\dotsc,s^2\}$.  Let~$A,B\subset V(G)$, and an assignment~$L_0$ of subsets of~$S$ of size~$s$ to
the vertices in~$B$ satisfy the conditions~(a), (b), and~(c) from the statement of Lemma~\ref{lemma-exab}.  There exists
an assignment~$L_1$ of subsets of~$S$ of size~$s$ to the vertices in~$A$ such that $|A_{\varphi,L_1}|>|B|$
for every $L_0$-coloring~$\varphi$ of~$G[B]$.
\end{lemma}
\begin{proof}
Choose for each vertex~$v\in A$ a set~$L_1(v)\subset S$ of size~$s$ uniformly at random,
each choice being independent of all the others.

Let~$\varphi$ be an arbitrary $L_0$-coloring of the vertices in~$B$.
For a vertex~$v$ in~$A$, let~$X_v$ be the set of colors that are
assigned by~$\varphi$ to at most~$k-1$ neighbors of~$v$ that belong to~$B$.
We assert that~$|X_v|<\lceil s^2/2\rceil$. Indeed, otherwise there exists a subset~$T$
of~$S$ of size~$\lceil s^2/2\rceil$ that is contained in~$X_v$.
According to the condition~(c) from the statement of Lemma~\ref{lemma-exab},
the vertex~$v$ has at least~$ks^2$ neighbors in~$B$ the lists of which are contained in~$T$.
By the pigeonhole principle, one of the (less than $s^2$) colors in~$T$ is assigned to more than~$k-1$ of these neighbors,
which contradicts the fact that~$T\subseteq X_v$. Therefore, $|X_v|<\lceil s^2/2\rceil$
for each vertex~$v\in A$.

This means that for every vertex~$v\in A$, the set $S\setminus X_v$ of colors that are
assigned by~$\varphi$ to at least~$k$ neighbors of~$v$ in~$B$ has size at least $s^2/2$.
If~$L_1(v)$ is a subset of $S\setminus X_v$, then no matter which color is chosen
for~$v$, there are at least~$k$ neighbors of~$v$ with the same color as~$v$.
The probability that $L_1(v)\subset S\setminus X_v$ is at least~$2^{-s-1}$ by the same calculation as in~\eqref{eq-estim}.

Recall that~$A_{\varphi,L_1}$ is the set of vertices~$v$ in~$A$ such that for every color~$c\in L_1(v)$,
the vertex~$v$ has at least~$k$ neighbors~$u$ in~$B$ with~$\varphi(u)=c$.
According to the previous paragraph and the condition~(a) from the statement of Lemma~\ref{lemma-exab}, the expected
size of the set~$A_{\varphi,L_1}$ is at least $|A|2^{-s-1}\ge n\cdot2^{-s-2}$.
Because all
choices for the lists of vertices in~$A$ are made independently, Chernoff's bound (Lemma~\ref{lemma:chernoff})
ensures that the probability that $|A_{\varphi,L_1}|\le n\cdot2^{-s-3}$
is at most~$e^{-n/2^{s+5}}$.  By the condition (b) from the statement of Lemma~\ref{lemma-exab} and the assumed lower bound on $d$, we have $|B|\le 2n/\sqrt{d}<n\cdot 2^{-s-3}$,
and thus the probability that $|A_{\varphi,L_1}|\le |B|$ is at most~$e^{-n/2^{s+5}}$.

There are $s^{|B|}$ $L_0$-colorings $\varphi$ of~$G[B]$, and thus the probability that $|A_{\varphi,L_1}|\le |B|$ for any of them is at most
\[s^{|B|}\cdot e^{-n/2^{s+5}} \le s^{2n/\sqrt{d}}\cdot e^{-n/2^{s+5}} = e^{-2n\left(2^{-s-6}-d^{-1/2}\ln s\right)},\]
which is less than~$1$ by the hypothesis on~$d$.  Hence, there exists a choice of~$L_1$ such that $|A_{\varphi,L_1}|>|B|$ for every $L_0$-coloring~$\varphi$ of~$G[B]$.
\end{proof}

\begin{theorem}\label{thm:minav}
Let~$s$ and~$k$ be positive integers, and let~$G$ be a graph of minimum degree at least~$d$.  If $d\ge 2^{14}k^2s^44^s$,
then $\chi^\ell_{\mathrm{mad},k}(G)>s$; that is, there exists an $s$-list assignment~$L$ for~$G$ such that any $L$-coloring
of~$G$ contains a monochromatic subgraph of average degree greater than~$k$.
\end{theorem}
\begin{proof}
Let~$G$ be a graph of minimum degree at least~$d\ge 2^{14}k^2s^44^s$.
Let~$S\coloneqq\{1,\dotsc,s^2\}$ be a set of colors, $n\coloneqq|V(G)|$, and let~$A,B\subset V(G)$ and~$L_0$ satisfy the conditions stated in Lemma~\ref{lemma-exab}.
Let~$L_1$ be an assignment of subsets of~$S$ of size~$s$ to the vertices in~$A$ such that
      $|A_{\varphi,L_1}|>|B|$ for every $L_0$-coloring~$\varphi$ of~$G[B]$,
which exists by Lemma~\ref{lemma:probsmall}.

Let~$L$ be any $s$-list assignment for~$G$ that extends~$L_0$ and~$L_1$.  We assert that any $L$-coloring of~$G$
contains a monochromatic subgraph of average degree greater than~$k$.  Indeed, consider such an $L$-coloring~$\psi$,
and let~$\varphi$ be the restriction of~$\psi$ to~$B$.
For each color~$c\in S$ set~$A_c\coloneqq \psi^{-1}(c)\cap A_{\varphi,L_1}$
and~$B_c\coloneqq \varphi^{-1}(c)\cap B$.
Since $|A_{\varphi,L_1}|>|B|$, there exists a color~$c\in S$ such that $|A_c|>|B_c|$.
Let~$H=G[A_c\cup B_c]$.  By the definition of~$A_{\varphi,L_1}$, each vertex~$v\in A_c$ has at least~$k$
neighbors in~$B_c$.  Let~$m$ be the number of edges of~$G$ between~$A_c$ and~$B_c$, so
      $m\ge k|A_c|$.
The average degree of the monochromatic subgraph~$H$ is $\tfrac{2|E(H)|}{|V(H)|}>\tfrac{2m}{2|A_c|}\ge k$,
as required.
\end{proof}

We are now ready to characterize the parameters satisfying~(CD).

\begin{proof}[Proof of Theorem~\ref{thm:cdchar}]
Let~$f$ be a hereditary graph parameter.  Suppose first that~$f$ bounds the average degree.
Let~$g\colon\mathbb{N}\to\mathbb{N}$ be a function such that every graph~$G$ has average degree at most~$g(f(G))$.
Without loss of generality, we can assume that~$g$ is non-decreasing.
Given positive integers~$p$ and~$s$, set~$k\coloneqq g(p)$ and~$s'\coloneqq 2^{14}k^2s^44^s$.  We assert that
$\col(G)\le s'$ for every graph~$G$ such that $\chi^{\ell}_{f,p}(G)\le s$.
Suppose, on the contrary, that $\col(G)\ge s'+1$, and thus~$G$ contains an induced subgraph~$G_0$
of minimum degree at least~$s'$.  Let~$L_0$ be an $s$-list assignment for~$G_0$ obtained using Theorem~\ref{thm:minav},
and let~$L$ be any extension of~$L_0$ to an $s$-list assignment of~$G$.
Since $\chi^{\ell}_{f,p}(G)\le s$, there exists an $(f,p)$-proper $L$-coloring $\varphi$ of~$G$.
Considering the restriction of~$\varphi$ to~$G_0$, Theorem~\ref{thm:minav} implies that~$G_0$ contains an induced subgraph~$H$
of average degree greater than~$k$ such that all vertices of~$H$ have the same color~$c$.
Since~$f$ bounds the average degree, $g(f(H))>k=g(p)$, and since~$g$ is non-decreasing, it follows that $f(H)>p$.
But since $f$ is hereditary, $f(G[\varphi^{-1}(c)])\ge f(H)>p$, contradicting the fact that~$\varphi$ is $(f,p)$-proper.
This contradiction implies that $\col(G)\le s'$, and thus~$f$ satisfies~(CD).

Suppose now that $f$ does not bound the average degree, and thus there exists some integer~$p$ and a sequence of
graphs~$(G_i)_{i\in\mathbb{N}}$
such that for every~$i$, the graph~$G_i$ has average degree at least~$i$ and $f(G_i)\le p$.
      Notice that any graph~$H$ has less than~$\col(H)|V(H)|$ edges, and thus average degree less than~$2\col(H)$.
It follows that $\col(G_i)>i/2$. Since $f$ is hereditary and $f(G_i)\le p$, we conclude that any coloring of~$G_i$ is $(f,p)$-proper,
and thus $\chi^{\ell}_{f,p}(G_i)=1$.  We deduce that $f$ does not satisfy~(CD) even if~$s$ is fixed to be~$1$.
\end{proof}

\section{Clustered coloring}\label{sec:cluster}

Let~$\star(G)$ be the maximum of the orders of the components of the graph~$G$. The parameter~$\chi_{\star,p}$ has been intensively
studied under the name \emph{clustered coloring}, and is among the most natural relaxations of the chromatic number.
As for some other variants (e.g., defective coloring), clustered coloring specialises to the usual notion
of vertex coloring: $\chi_{\star,1}(G)=\chi(G)$.
Clustered colorings appeared in a variety of contexts and it seems that the first published work is one dealing
with databases~\cite{KMRV97}.

The parameter~$\star$ bounds the average degree, since the average degree of a
graph~$G$ is at most~$\star(G)-1$. Therefore, the proof of Theorem~\ref{thm:cdchar} implies
that $\col_{\star,1}(G)=\col(G)\le 2^{14}p^2s^44^s$ for
every graph~$G$ such that $\chi^{\ell}_{\star,p}(G)\le s$.  In the context of the property~(CC), it is natural to ask whether it is
possible to bound~$\col_{\star,p'}(G)$ for a sufficiently large value of~$p'=p'(p,s)$ by a function depending only on~$s$
(i.e., whether all effects of allowing large clusters for the list chromatic number cannot be absorbed by allowing even larger
clusters for the coloring number).  Alternately, one can wonder whether allowing large values for~$p'$ cannot enable us to substantially
improve the dependence of~$\col_{\star,p'}(G)$ on~$s$.

This motivates the study of the following function.  Let~$s$ and~$p$
be integers.  We define~$h_\star(p,s)$ to be the smallest integer~$s'$ such that for some~$p'$, all graphs~$G$ with $\chi^{\ell}_{\star,p}(G)\le s$
satisfy $\col_{\star,p'}(G)\le s'$.  Hence, $h_\star(p,s)\le 2^{14}p^2s^44^s$, and we ask whether~$h_\star(p,s)$ can be bounded from
above by a function of~$s$ only, or by a function of~$p$ and~$s$ that is subexponential in~$s$.  We answer both of these questions
negatively: the first question through Lemma~\ref{lemma-path} and the second one through Corollary~\ref{cor-exp}.

\begin{figure}
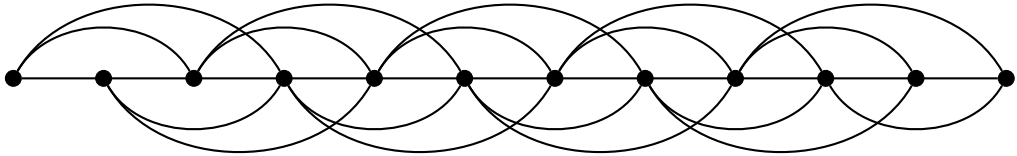

\begin{center}
      \begin{tikzgraph}
\def\l{12mm}
\foreach \i in {0,...,11}
  \draw (\i*\l,\l) node[vertex](u\i){};

\foreach \i in {0,2,...,8}
  \foreach \j in {2,3}
     \draw[edge] let \n1={int(\i+\j)} in (u\i) to[out = 60, in = 120] (u\n1);
\foreach \i in {1,3,...,7}
  \foreach \j in {2,3}
     \draw[edge] let \n1={int(\i+\j)} in (u\i) to[out = -60, in = -120] (u\n1);

\draw[edge] (u9) to[out = -60, in = -120] (u11);
\draw[edge] (u0)--(u11);
\end{tikzgraph}
\end{center}
      \caption{The graph~$P_{12}^3$.}\label{fig-2}
\end{figure}

Let~$P^t_n$ be the $t$-th distance power of a path~$P_n$ on~$n$ vertices, that is, a graph on vertex set $\{1,\dotsc, n\}$
where two vertices~$i$ and~$j$ are adjacent if and only if~$1\le |i-j|\le t$, see Figure~\ref{fig-2}.
Two vertices of~$P^t_n$ are \emph{consecutive} if they are adjacent in~$P_n$.

\begin{lemma}\label{lemma-path}
If~$t$ and~$n$ are positive integers, then
$\chi^\ell_{\star,2t^2}(P^t_n)\le 2$.
      Also, if~$p'$ is a positive integer and~$n \ge p'(t+1)+t+1$,
      then $\col_{\star,p'}(P^t_n)\ge t+1$.  Consequently, for every positive integer~$p$ and every
      integer~$s\ge 2$,
\[
      h_\star(p,s)\ge h_\star(p,2)\ge \lfloor \sqrt{p/2}\rfloor+1.
\]
\end{lemma}

\begin{proof}
      To prove that $\chi^\ell_{\star,2t^2}(P^t_n)\le 2$, we may assume
without loss of generality that~$n$ is a multiple of~$t(t+1)$. Let~$L$
      be a $2$-list assignment for~$P_n^t$. We
      split~$P^t_n$ into blocks~$B_1,\dotsc,B_m$ of~$t(t+1)$ consecutive vertices and we $L$-color each
      block independently.  To color the block~$B_k$, we further split it into $t$-tuples of
      consecutive vertices~$T_0^k,T_1^k,\dotsc,T_t^k$.
      (Notice, for later, that there is no edge between~$T_i^k$ and~$T_j^k$
      whenever~$|i-j|>1$.) 
      First we $L$-color arbitrarily the vertices
      in $T_0^k$ and then for each~$i\in\{1,\dotsc,t\}$ we $L$-color the vertices in~$T_i^k$ without using
      the color of the $i$-th vertex of~$T_0^k$. This can be achieved as every list has size at least~$2$.
      Now consider a monochromatic component~$H$ of~$P^t_n$. In each block~$B_k$,
      there is a $t$-tuple~$T_i^k$ that is disjoint from~$V(H)$. Indeed,
      if~$V(H)$ intersects~$T_0^k$, then the color used
      on the vertices in~$H$ is not assigned to the vertices in one of the other $t$-tuples of~$B_k$.
      This property further implies that~$H$ is contained in at most two blocks. 
      We thus infer that~$H$ has at most~$2t^2$ vertices. 

      Now let~$p'$ be a positive integer and further assume that $n\ge p'(t+1)+t+1$.
Let~$I$ be a non-empty subset of vertices of~$P^t_n$ of size at most~$p'$.
      Let~$W$ be a set of consecutive vertices of~$P^t_n$ of maximum size
      under the constraint that $W \cap I = \varnothing$. The assumption on~$n$ implies that
      $|W| \geq t+1$. It follows that there exists a vertex~$v$ in~$I$
      that has at least $t$ neighbors contained in~$W$
      (simply consider a vertex directly preceding or following~$W$ in the order given
      by the path~$P_n$).
      Therefore~$I$ is not a $t$-island, and thus $\col_{\star,p'}(P^t_n)\ge t+1$.
\end{proof}

As for the exponential (in~$s$) lower bound on~$h_\star(p,s)$,
we combine the probabilistic construction by Alon and Krivelevich~\cite{AlKr98}
with the following observation.

\begin{lemma}\label{lemma:coldens}
Let~$p$ be an integer and~$\alpha$ a non-negative real number. Let~$G$ be a graph such that every subgraph~$H$ of~$G$
with at most~$p$ vertices has at most~$\alpha|V(H)|$ edges.  Then $G$ has average degree less than
$2(\col_{\star,p}(G)+\alpha)$.
\end{lemma}

\begin{proof}
      Set~$s \coloneqq \col_{\star,p}(G)$.
      We prove by induction on the number of vertices of~$G$ that
\[
      |E(G)| \leq |V(G)| \cdot (\alpha + s-1),
\]            
      the statement being trivially true if~$V(G)$ is empty. Suppose that $V(G)$ is not empty;
      then~$G$ has an~$s$-island inducing a subgraph with all components of size at most~$p$.
      Let~$I$ be the vertex set of one such component, and note that~$I$ is also an $s$-island.
      Since $|I|\le p$, we have $|E(G[I])| \leq \alpha|I|$ by the assumptions.
      Since $I$ is an $s$-island, there are at most $(s-1)|I|$ edges between~$I$ and~$V(G)\setminus I$.
     Consequently $|E(G)| \leq |E(G-I)| + |I| \cdot (\alpha + s-1)$, and since $|E(G-I)|\le |V(G-I)| \cdot (\alpha + s-1)$
     by the induction hypothesis, we have $|E(G)| \leq |V(G)| \cdot (\alpha + s-1)$.

     Consequently, the average degree of $G$ is
     $$2|E(G)|/|V(G)|\le 2(\alpha+s-1)<2(\col_{\star,p}(G)+\alpha).$$
\end{proof}

We use Lemma~\ref{lemma:coldens} in conjunction with a result of Alon
and Krivelevich~\cite[Proposition~2.2]{AlKr98}. Let~$\GG(n,n,p)$ be
the probability space of bipartite graphs~$G$ with both parts of
their bipartition of order~$n$, where the probability of obtaining
any given bipartite graph~$G$ with~$n$ vertices in each part
is~$p^{|E(G)|}(1-p)^{n^2-|E(G)|}$. In other words, each of the~$n^2$
possible edges is chosen to belong to~$G$ with probability~$p$
independently at random.

\begin{lemma}\label{lemma:randens}
For every positive real number~$C$, there exists an integer~$d_0\ge 3$ such that the following holds for every $d\ge d_0$. If~$G\in \GG(n,n,d/n)$, then
w.h.p.\ (as~$n$ goes to infinity), each subgraph~$H$ of~$G$ with at most
$Cn/d$ vertices in each part has at most $\frac{3\log_2 d}{\log_2\log_2 d}|V(H)|$ edges.
\end{lemma}

\noindent
Consequently, the following is true.

\begin{corollary}\label{cor:density}
There exists $d_0\ge 153$ such that the following holds.
Let~$d$ and~$p$ be positive integers. If~$d\ge d_0$ and~$G\in \GG(n,n,d/n)$,
      then w.h.p.\ (as $n$ goes to infinity) $\col_{\star,p}(G)>d/5$.
\end{corollary}

\begin{proof}
      On one hand, the expected number of edges of~$G$ is~$dn$, and
Lemma~\ref{lemma:chernoff} implies that $|E(G)|\ge dn/2$ w.h.p.,
and since $|V(G)|=2n$, the average degree of~$G$ is at least~$d/2$ w.h.p.
On the other hand, Lemma~\ref{lemma:randens} implies that
w.h.p.\ every subgraph~$H$ of~$G$ with at most~$n/d$ vertices
contains at most~$\frac{3\log_2 d}{\log_2\log_2 d}|V(H)|$
edges. Consequently, for any $n\ge pd$, Lemma~\ref{lemma:coldens} applies
to~$G$ with $\alpha=\frac{3\log_2 d}{\log_2\log_2 d}$, which
yields that the average degree of~$G$ is less
than~$2\left(\col_{\star,p}(G)+\frac{3\log_2 d}{\log_2\log_2
d}\right)$. We thus infer that $\col_{\star,p}(G)\ge d/4-\frac{3\log_2
d}{\log_2\log_2 d}$, which is greater than~$d/5$ since~$d\ge153$.
\end{proof}

Alon and Krivelevich~\cite[Theorem~1.1]{AlKr98} proved that w.h.p., graphs in
$\GG(n,n,d/n)$ have choosability $(1+o(1))\log_2 d$.  Together with
Corollary~\ref{cor:density}, this provides an exponential (in~$s$) lower
bound on~$h_\star(p,s)$.

\begin{corollary}\label{cor-exp}
      For any positive integer $s$,
      \[h_\star(1,s)\ge \frac15\cdot2^{(1+o(1))s}.\]
\end{corollary}

\begin{proof}
Consider a random graph~$G \in \GG(n,n,d/n)$. As reported earlier,
      $\chi_{\star,1}^\ell(G)=\chi^\ell(G) = (1+o(1))\log_2 d$ w.h.p.,
      which means that $d = 2^{(1+o(1))\chi^\ell_{\star,1}(G)}$.
      In addition, Corollary~\ref{cor:density} implies that for any positive integer $p'$, $\col_{\star,{p'}}(G)\ge d/5$ w.h.p.
      Altogether we obtain $\col_{\star,{p'}}(G) \ge 2^{(1+o(1))\chi_{\star,1}^\ell(G)}/5$ w.h.p.
      Since this holds for any~$p'$, we conclude that $h_\star(1,s)\ge \frac15\cdot2^{(1+o(1))s}$.
\end{proof}
Since $h_\star(p,s)\ge h_\star(1,s)$ for all positive integers $p$, this gives another lower bound on~$h_\star(p,s)$.

Let us note here that an inspection of the proof of Alon and
Krivelevich~\cite[Theorem~1.1]{AlKr98} also gives for graphs in $\GG(n,n,d/n)$
and w.h.p., an upper bound on~$\chi^\ell_{\star,p}(G)$ that does not depend on~$p$,
namely~$(1+o(1))\log_2 d$. More generally, the same
phenomenon can be seen in the setting of Theorem~\ref{thm:cdchar} for
graphs of large girth. We use the following well-known observation.

\begin{lemma}\label{lemma-mindeg}
Let~$g$ be an odd integer greater than one and let~$k$ be a positive integer.
Each graph~$H$ of girth at least~$g$ and average degree at least~$2k$ has a component
with more than~$(k-1)^{(g-1)/2}$ vertices.
\end{lemma}

\begin{proof}
We can assume that~$H$ is connected.
Let~$H'$ be a minimal subgraph of~$H$ of average degree at least~$2k$.
Then~$H'$ has minimum degree at least~$k$. Let~$v$ be any vertex of~$H'$.
Let $H''$ be the subgraph of $H'$ induced by vertices at distance at most~$(g-1)/2$ from~$v$.
Since $H$ has girth at least~$g$, the graph~$H''$ is a tree.  Furthermore, all vertices
of $H''$ at distance less than $(g-1)/2$ from $v$ have degree at least $k$.
Consequently,
\[
|V(H)| \ge |V(H'')|\ge k(k-1)^{(g-3)/2}>(k-1)^{(g-1)/2}.
\]
\end{proof}

Combining the bound on the component size provided by Lemma~\ref{lemma-mindeg}
with Theorem~\ref{thm:minav}, we obtain the following upper bound on the coloring
number.

\begin{theorem}\label{thm:girth}
Let~$s$ and~$p$ be positive integers with~$p\ge 2$. Let~$g$ be an odd integer greater than one.
Set $k\coloneqq1+\lceil p^{2/(g-1)}\rceil$ and $d\coloneqq 2^{16}k^2s^44^s$.
If a graph~$G$ of girth at least~$g$ satisfies $\chi^\ell_{\star,p}(G)\le s$,
then~$\col(G)\le d$.  In particular, if $g\ge 2\lceil\log_2 p\rceil+1$, then $\col(G)\le 9\cdot 2^{16}s^44^s$.
\end{theorem}

\begin{proof}
We establish the contrapositive: letting~$G$ be a graph of girth at least~$g$
such that~$\col(G)>d$, we prove that $\chi^\ell_{\star,p}(G)>s$.
By replacing $G$ by a minimal subgraph of~$G$ with coloring number greater than~$d$ if necessary,
we can without loss of generality assume that~$G$ has minimum degree at least~$d$.
Theorem~\ref{thm:minav} then ensures the existence of
an $s$-list assignment~$L$ for~$G$ such
that any $L$-coloring of~$G$ contains a monochromatic subgraph~$H$
of average degree greater than~$2k$. By Lemma~\ref{lemma-mindeg}, the
subgraph~$H$ has more than~$(k-1)^{(g-1)/2}\ge p$ vertices, and since
such a subgraph exists for any $L$-coloring of~$G$, we deduce that
$\chi^\ell_{\star,p}(G) > s$.
\end{proof}

Since~$\col(G)$ bounds the choosability of~$G$, the proof of Theorem~\ref{thm:cdchar} implies 
that $\chi^\ell(G)\le 2^{14}p^2s^44^s$ for every graph~$G$ such that $\chi^\ell_{\star,p}(G)\le s$.  However,
we are not able to find examples of graphs for which~$\chi^\ell$ and~$\chi^\ell_{\star,p}$ would be far apart.
As far as we are aware, the answer to the following question could be positive.

\begin{question}\label{q-clust}
Is~$\chi^\ell(G)$ at most~$p\cdot\chi^\ell_{\star,p}(G)$ for every graph~$G$ and
positive integer~$p$?
\end{question}

Note that for the usual chromatic number, $\chi(G)\le
p\cdot\chi_{\star,p}(G)$ holds, since we can first color~$G$ so that
monochromatic components have size at most~$p$, next replace each color
by~$p$ new colors and use them to properly color the vertices in each
monochromatic components to obtain a proper coloring of~$G$.

Furthermore, Question~\ref{q-clust} is a weakening of
an analogous question for defective coloring (whether $\chi^\ell(G)\le (p+1)\cdot\chi^\ell_{\Delta,p}(G)$),
which has been mentioned as folklore~\cite{Kan13}.  Given Theorems~\ref{thm:cdchar} and~\ref{thm:minav}, it is also
natural to ask the following stronger question.

\begin{question}
Is~$\chi^\ell(G)$ at most~$(p+1)\cdot\chi^\ell_{\mathrm{mad},p}(G)$ for every graph~$G$ and
positive integer~$p$?
\end{question}

Again, it is easy to see that $\chi(G)\le (p+1)\cdot\chi_{\mathrm{mad},p}(G)$ holds, since
every graph of maximum average degree at most $p$ is $(p+1)$-colorable.


\begin{thebibliography}{10}

\bibitem{Alo00}
{\sc N.~Alon}, {\em Degrees and choice numbers}, Random Structures Algorithms,
  16 (2000), pp.~364--368.

\bibitem{ADOV03}
{\sc N.~Alon, G.~Ding, B.~Oporowski, and D.~Vertigan}, {\em Partitioning into
  graphs with only small components}, J. Combin. Theory Ser. B, 87 (2003),
  pp.~231--243.

\bibitem{AlKr98}
{\sc N.~Alon and M.~Krivelevich}, {\em The choice number of random bipartite
  graphs}, Ann. Comb., 2 (1998), pp.~291--297.

\bibitem{BDN18}
{\sc E.~Berger, Z.~Dvo\v{r}{\'a}k, and S.~Norin}, {\em Treewidth of grid
  subsets}, Combinatorica, 36 (2018), pp.~1337--1352.

\bibitem{CCW86}
{\sc L.~J. Cowen, R.~H. Cowen, and D.~R. Woodall}, {\em Defective colorings of
  graphs in surfaces: partitions into subgraphs of bounded valency}, J. Graph
  Theory, 10 (1986), pp.~187--195.

\bibitem{CGJ97b}
{\sc L.~J. Cowen, W.~Goddard, and C.~E. Jesurum}, {\em Defective coloring
  revisited}, J. Graph Theory, 24 (1997), pp.~205--219.

\bibitem{CuKi10}
{\sc W.~Cushing and H.~A. Kierstead}, {\em Planar graphs are 1-relaxed,
  4-choosable}, European J. Combin., 31 (2010), pp.~1385--1397.

\bibitem{DDO+04}
{\sc M.~DeVos, G.~Ding, B.~Oporowski, D.~P. Sanders, B.~Reed, P.~Seymour, and
  D.~Vertigan}, {\em Excluding any graph as a minor allows a low tree-width
  2-coloring}, J. Combin. Theory Ser. B, 91 (2004), pp.~25--41.

\bibitem{DvNo17}
{\sc Z.~Dvo\v{r}{\'a}k and S.~Norin}, {\em Islands in minor-closed classes.
  {I}. {B}ounded treewidth and separators}, arXiv, 1710.02727 (2017).

\bibitem{EsOc16}
{\sc L.~Esperet and P.~Ochem}, {\em Islands in graphs on surfaces}, SIAM J.
  Discrete Math., 30 (2016), pp.~206--219.

\bibitem{FrHe94}
{\sc M.~Frick and M.~A. Henning}, {\em Extremal results on defective colorings
  of graphs}, Discrete Math., 126 (1994), pp.~151--158.

\bibitem{FrKa16}
{\sc A.~Frieze and M.~Karo\'nski}, {\em Introduction to random graphs},
  Cambridge University Press, Cambridge, 2016.

\bibitem{HaSe06}
{\sc F.~Havet and J.-S. Sereni}, {\em Improper choosability of graphs and
  maximum average degree}, J. Graph Theory, 52 (2006), pp.~181--199.

\bibitem{HST03}
{\sc P.~Haxell, T.~Szab\'o, and G.~Tardos}, {\em Bounded size
  components---partitions and transversals}, J. Combin. Theory Ser. B, 88
  (2003), pp.~281--297.

\bibitem{Kan13}
    {\sc R.~J. Kang}, {\em Improper choosability and Property~{B}},
J. Graph Theory, 73 (2013), pp.~342--353.

\bibitem{KaMc10}
{\sc R.~J. Kang and C.~McDiarmid}, {\em The {$t$}-improper chromatic number of
  random graphs}, Combin. Probab. Comput., 19 (2010), pp.~87--98.

\bibitem{KMRV97}
{\sc J.~Kleinberg, R.~Motwani, P.~Raghavan, and S.~Venkatasubramanian}, {\em
  Storage management for evolving databases}, in 38th Annual Symposium on
  Foundations of Computer Science (FOCS '97), IEEE, 1997, pp.~353--362.

\bibitem{Skr99b}
{\sc R.~{\v{S}}krekovski}, {\em A {G}r{\"o}tzsch-type theorem for list
  colourings with impropriety one}, Combin. Probab. Comput., 8 (1999),
  pp.~493--507.

\bibitem{Woo17}
{\sc D.~R. Wood}, {\em Defective and clustered graph colouring}, Electron. J. Combinatorics, \#DS23, 2018.

\bibitem{Woo04}
{\sc D.~R. Woodall}, {\em Defective choosability of graphs with no
  edge-plus-independent-set minor}, J. Graph Theory, 45 (2004), pp.~51--56.

\end{thebibliography}
\end{document}